\documentclass[a4paper,11pt,reqno]{amsart}

\usepackage{graphicx}
\usepackage{mathtools}
\usepackage{hyperref}
\usepackage{float} 
\usepackage{amsfonts}
\usepackage{amsthm}
\usepackage{latexsym}
\usepackage{amsmath}
\usepackage{amssymb}
\usepackage{amscd}
\usepackage{epsf}
\usepackage{tikz}
\usetikzlibrary{matrix,arrows}

\usepackage[english]{babel}
\usepackage{lmodern}

\usepackage[utf8]{inputenc}

\usepackage[T1]{fontenc}
\newtheorem{thm}{Theorem}

\newcommand{\nc}{\newcommand}
\nc{\Ra}{\Rightarrow}
\nc{\ra}{\rightarrow}
\nc{\ora}{\overrightarrow}
\nc{\ol}{\overline}
\nc{\de}{\delta}
\nc{\De}{\Delta}
\nc{\p}{\partial}
\nc{\ex}{\exists}
\nc{\nex}{\nexists}
\nc{\fa}{\forall}
\nc{\Ea}{\Leftrightarrow}
\nc{\ea}{\leftrightarrow}
\nc{\ub}{\underbrace}
\nc{\Rn}{\mathbb{R}^n}
\nc{\Rne}{\mathbb{R}^{n+1}}
\nc{\Rnm}{\mathbb{R}^{n-1}}
\nc{\bR}{\mathbb{R}}
\nc{\bC}{\mathbb{C}}
\nc{\bN}{\mathbb{N}}
\nc{\bZ}{\mathbb{Z}}
\nc{\bfx}{\textbf{\textit{x}}}
\nc{\bfX}{\textit{\textbf{X}}}
\nc{\bfY}{\textbf{\textit{Y}}}
\nc{\bfZ}{\textbf{\textit{Z}}}
\nc{\bfy}{\textbf{\textit{y}}}
\nc{\bfa}{\textbf{\textit{a}}}
\nc{\bfv}{\textbf{\textit{v}}}
\nc{\D}{\Delta}
\nc{\Dn}{\Delta^n}
\nc{\lx}{\lambda(x)}
\nc{\ba}{\textbf{\textit{a}}}
\nc{\Sn}{\S^n}
\nc{\Snm}{\S^{n-1}}
\nc{\Sne}{\S^{n+1}}
\nc{\na}{\nabla}
\nc{\deri}[2]{\frac{d #1}{d #2}}
\nc{\ph}[1]{\phi\left( #1\right)}
\nc{\ft}{_{n=1}^\infty}
\nc{\gs}{\sum_{n=1}^\infty}
\nc{\qf}[2]{\langle #1 #2,#2\rangle}
\nc{\mc}[1]{\mathcal{#1}}
\nc{\conj}[1]{\overline{#1}}
\everymath{\displaystyle}

\title{Alternative proofs of Mandrekar's theorem}
\author{Linus Bergqvist}
\address{Department of Mathematics, Stockholm University, 106 91 Stockholm, Sweden.}
\email{linus@math.su.se}
\date{}


\begin{document}

\begin{abstract}
We present two alternative proofs of Mandrekar's theorem, which states that an invariant subspaces of the Hardy space on the bidisc is of Beurling type precisely when the shifts satisfy a doubly commuting condition. The first proof uses properties of Toeplitz operators to derive a formula for the reproducing kernel of certain shift invariant subspaces, which can then be used to characterize them. The second proof relies on the reproducing property in order to show that the reproducing kernel at the origin must generate the entire shift invariant subspace.
\end{abstract}

\maketitle

\section{Background}

In this note we give two alternative proofs of Mandrekar's theorem on shift invariant subspaces of the Hardy space on the bidisc. For the first proof we will use the main idea from the alternative proof of Beurling's theorem given by Karaev in \cite{Karaev}, and use it to characterize certain closed, shift invariant subspaces through their reproducing kernels. The second proof is a modified version of the classical proof of Beurling's theorem given in \cite{Hoffman}. We can already here point out that the proofs given in this article work just as well for $H^2(\mathbb{D}^n)$, though certain calculations become a bit more involved.

As usual, we say that a closed subspace $\mathcal{M} \subset H^2(\mathbb{D}^2)$ is \emph{shift invariant} if $S_j \mathcal{M} \subset \mathcal{M}$ for $j=1,2$, where
$$
S_j: H^2(\mathbb{D}^2) \mapsto H^2(\mathbb{D}^2), \quad f(z) \mapsto z_j f(z), \quad \text{for } j=1,2.
$$

We want to prove the following.

\begin{thm}[Theorem 2 from \cite{Mandrekar}] \label{mandrekar}

An invariant subspace $\mathcal{M} \neq \{0\}$ of $H^2(\mathbb{D}^2)$ is of the form $\varphi H^2$ with $\varphi$ an inner function if and only if the shift operators $S_1$ and $S_2$ are \emph{doubly commuting} on $\mathcal{M}$.
\end{thm}

That $\{S_j\}_{j=1,2}$ is doubly commuting means that the operators commute with each other, and with each other's adjoints; that is $S_i S_j = S_j S_i$ and $S_i S_j^* = S_j^* S_i$ for $i \neq j$. As always, a function $\varphi$ is inner if $\varphi \in H^\infty(\mathbb{D}^2)$ and $|\varphi(z)| = 1$ almost everywhere on $\mathbb{T}^2$.

Since we will characterize closed invariant subspaces through their reproducing kernels, we will begin by going through some general theory for the reproducing kernels of operator range spaces and how Toeplitz operators act on such spaces. For more details see \cite{Model} and chapter $1$ of \cite{sub-Hardy}.

A Hilbert space $H$ with inner product $\langle \cdot, \cdot \rangle_H$ consisting of functions defined on a domain $D$ is called a reproducing kernel Hilbert space if point evaluations are bounded linear functionals. By Riesz representation theorem, this means that for every point $z_0 \in D$ there is a function $k_{z_0} \in H$, called the \emph{reproducing kernel at $z_0$}, such that 
$$
f(z_0) = \langle f, k_{z_0} \rangle_H
$$
for all $f \in H$. The reproducing kernel of $H^2(\mathbb{D}^n)$ is the Cauchy kernel
$$
C_\lambda(z) = \prod_{j=1}^n \frac{1}{(1-\overline{\lambda_j}z_j)},
$$
and throughout this note, we will denote by $k_\lambda^{\mathcal{M}}$ the reproducing kernel of a closed subspace $\mathcal{M} \subset H^2(\mathbb{D}^n)$.

If $T: H \mapsto H$ is a partial isometry on a reproducing kernel Hilbert space $H$ with reproducing kernel $k_\lambda(z)$, then $T(H)$ and $T(H)^\perp$ are closed subspaces of $H$, and their reproducing kernels are given by $T T^* k_\lambda(z)$ and $(1- T T^*)k_\lambda(z)$ respectively. In particular, we are interested in Toeplitz operators
$$
T_\varphi: H \mapsto H, \quad f \mapsto \varphi f,
$$ 
with symbol $\varphi$ belonging the multiplier algebra of $H$. In the case of $H^2(\mathbb{D}^n)$, the multiplier algebra is $H^\infty(\mathbb{D}^n)$, the space of bounded analytic functions on $\mathbb{D}^n$. 

If $T_\varphi$ is an isometry, then the above means that $\varphi H$ and $(\varphi H)^\perp$ will have reproducing kernels
\begin{equation}\label{kernels}
\varphi(z) \overline{\varphi(\lambda)} k_\lambda(z) \quad \text{ and } \quad  (1 - \varphi(z) \overline{\varphi(\lambda)}) k_\lambda(z)
\end{equation}
respectively. In order to arrive at \eqref{kernels} we have used that
\begin{equation} \label{Toeplitz_and_kernels}
T_{\varphi}^* = T_{\overline{\varphi}} \quad \text{ and } \quad T_{\overline{\varphi}} k_\lambda(z) = \overline{\varphi(\lambda)} k_\lambda(z).
\end{equation}
These two equalities always hold for Toeplitz operators on reproducing kernel Hilbert spaces, and will be used frequently in this note.

Note that $T_\varphi$ is an isometry whenever $\varphi$ is an inner function and $H$ is the Hardy space or a closed subspace of the Hardy space, and thus \eqref{kernels} holds in these cases.

\section{First proof}

The new idea in the proof given here is that we show that an invariant subspace on which the shift operators are doubly commuting must be of the form $\varphi H^2$ for $\varphi$ inner by characterizing the reproducing kernel of such a subspace, instead of using the wandering subspace theorem from \cite{Wold} as is done in the original proof. For some other parts of the theorem, we will still need Mandrekar's original proof. For completeness and to get a self-contained proof, we will present the proofs from Mandrekar's article of certain statements that we need to use here as well, and furthermore we will add proofs of statements from \cite{Wold} which are used in the proof.

\begin{proof}[First proof of Theorem~\ref{mandrekar}]

First of all, if $\mathcal{M} = \varphi H^2$ for an inner function $\varphi$, then $S_1$ and $S_2$ are commuting on $\mathcal{M}$ since $z_1 z_2 \varphi(z) f(z) = z_2 z_1 \varphi(z) f(z)$ for all $f \in H^2$. Then Mandrekar refers to Theorem $1$ $(i) \Ra (ii)$ in \cite{Wold} to show that $S_1$ and $S_2$ are in fact doubly commuting on such a subspace $\mathcal{M}$ if $\varphi$ is inner (or more generally, of constant modulus on $\mathbb{T}^2$).

Another way of showing that $S_1$ and $S_2$ are doubly commuting on such a subspace $\mathcal{M}$ is by considering their action on the reproducing kernel. From the first formula of \eqref{kernels} we know that the reproducing kernel of $\mathcal{M}$ is
$$
k^{\mathcal{M}}_\lambda(z) = \frac{\overline{\varphi(\lambda)} \varphi(z)}{(1-\overline{\lambda_1}z_1)(1-\overline{\lambda_2}z_2)}.
$$
Since the reproducing kernels are dense in $\mathcal{M}$, we only need to show that
$$
S_1 S_2^* k^{\mathcal{M}}_\lambda(z) = S_2^* S_1 k^{\mathcal{M}}_\lambda(z)
$$
for the kernels. By again using equation \eqref{Toeplitz_and_kernels}, we immediately see that the left hand side is $z_1 \overline{\lambda_2} k^{\mathcal{M}}_\lambda(z)$. To see that the same holds for the right hand side we will use that $k^{\mathcal{M}}_{\lambda}(z) = \varphi(z) \overline{\varphi(\lambda)} C_{\lambda_1}(z_1) C_{\lambda_2}(z_2)$, to get
\begin{multline*}
S_2^*(z_1 k^{\mathcal{M}}_{\lambda}(z)) = \left\langle w_1 \overline{w_2} k_{\lambda}^{\mathcal{M}}(w), k_z^{\mathcal{M}}(w) \right\rangle_{H^2(\mathbb{D}^2)} \\
= \left\langle w_1 \overline{w_2} \varphi(z) \overline{\varphi(\lambda)} C_{\lambda_1}(z_1) C_{\lambda_2}(z_2), \varphi(z) \overline{\varphi(\lambda)} C_{\lambda_1}(z_1) C_{\lambda_2}(z_2) \right\rangle_{H^2(\mathbb{D}^2)} \\
= \left\langle \varphi(z) \overline{\varphi(\lambda)} w_1 C_{\lambda_1}(w_1) \langle \overline{w_2} C_{\lambda_2}(w_2), C_{z_2}(w_2)\rangle_{H^2(\mathbb{D})}, C_{z_1}(w_1) \right\rangle_{H^2(\mathbb{D})} \\
= \left\langle \varphi(z) \overline{\varphi(\lambda)} w_1 C_{\lambda_1}(w_1) \overline{\lambda_2} C_{\lambda_2}(z_2), C_{z_1}(w_1) \right\rangle_{H^2(\mathbb{D})} \\
= \varphi(z) \overline{\varphi(\lambda)} \overline{\lambda_2} C_{\lambda_2}(z_2) \left\langle w_1 C_{\lambda_1}(w_1), C_{z_1}(w_1) \right\rangle_{H^2(\mathbb{D})} \\
= \varphi(z) \overline{\varphi(\lambda)} \overline{\lambda_2} C_{\lambda_2}(z_2) z_1 C_{\lambda_1}(z_1) = z_1 \overline{\lambda_2} k^{\mathcal{M}}_\lambda(z).
\end{multline*}
Note that in the second equality we have used that $\varphi$ is inner in order to get rid of a factor $|\varphi(w)|^2$, and in the third equality we have used \eqref{Toeplitz_and_kernels} to conclude that $\langle \overline{w_2} C_{\lambda_2}(w_2), C_{z_2}(w_2)\rangle_{H^2(w_2 \in \mathbb{T})} = \overline{\lambda_2} C_{\lambda_2}(z_2)$.

This proves one direction of the theorem. It remains to show the other direction.

One proof that we will essentially get from \cite{Mandrekar} is the claim that if $\mathcal{M}$ is an invariant subspace on which $\{S_j\}_{j=1,2}$ is doubly commuting, then the closed subspace $O_1(\mathcal{M}) \cap O_2(\mathcal{M})$ either has dimension $1$ and contains an inner function, or it only contains the zero function. Here
$$
O_j(\mathcal{M}) = \mathcal{M} \ominus S_j \mathcal{M}.
$$

In order to show that the dimension is smaller than or equal to $1$, we must use the fact that
\begin{equation} \label{commutativeContainment}
S_2(O_1) \subset O_1.
\end{equation}
Mandrekar proves this by simply referring to Theorem $1$ $(iii)$ in \cite{Wold}. However, this can be shown with elementary arguments using that $\{S_j\}_{j=1,2}$ are doubly commuting on $\mathcal{M}$. Namely, since $S_1$ and $S_2^*$ are commuting, we have that for every $g \in \mathcal{M} \ominus S_1 \mathcal{M}$
$$
0 = \langle g, S_1 (S_2^* f) \rangle = \langle g, S_2^* (S_1  f) \rangle = \langle S_2 g, S_1 f \rangle
$$
for all $f \in \mathcal{M}$, and so $S_2 g \in \mathcal{M} \ominus S_1 \mathcal{M}$. Again, here we regard $S_j$ and $S_j^*$ as operators on $\mathcal{M}$.

With this result at hand, we can go on with presenting the way Mandrekar proves that the dimension of $O_1(\mathcal{M}) \cap O_2(\mathcal{M})$ is $1$ unless the intersection only contains the zero function.

Let $g_1, g_2 \in O_1(\mathcal{M}) \cap O_2(\mathcal{M})$. Then for all $m,n > 0$
\begin{equation*}
\int_{\mathbb{T}^2} z_1^m z_2^n g_1(z) \overline{g_2(z)} dz = 0,
\end{equation*}
and by \eqref{commutativeContainment} we also have that
\begin{equation*}
\int_{\mathbb{T}^2} z_2^n g_1(z) \overline{z_1^m g_2(z)} dz = 0
\end{equation*}
for all $m,n > 0$.
By symmetry and since $\overline{z_1}^m = z_1^{-m}$, this means that
\begin{equation} \label{g_1g_2 is constant}
\int_{\mathbb{T}^2} z_1^m z_2^n g_1(z) \overline{g_2(z)} dz = 0,
\end{equation}
for all $(m,n) \neq (0,0)$, which means that $g_1(z) \overline{g_2(z)} = c$ a.e. on $\mathbb{T}^2$. 

Now suppose that there are $g_1, g_2 \in O_1(\mathcal{M}) \cap O_2(\mathcal{M}) \setminus \{0 \}$ with $g_1 \perp g_2$. 
In this case $|g_1|^2 = c_1 \neq 0$ and $|g_2|^2 = c_2 \neq 0$ a.e. on $\mathbb{T}^2$ by \eqref{g_1g_2 is constant}, but $g_1 \overline{g_2} = 0$ a.e. on $\mathbb{T}^2$ since $g_1 \overline{g_2}$ is a.e. constant and $g_1 \perp g_2$. 

This is impossible, and hence $O_1(\mathcal{M}) \cap O_2(\mathcal{M})$ is one-dimensional. Furthermore, from the above arguments it follows that $O_1(\mathcal{M}) \cap O_2(\mathcal{M})$ contains, and hence be generated by an inner function, which we denote by $\varphi$. All that remains is to show that this implies that $\mathcal{M} = \varphi H^2(\mathbb{D}^2)$ for this inner function $\varphi$.

The next step in Mandrekar's proof relies on a wandering subspace theorem for commuting isometries due to Słociński from \cite{Wold}. Here, we instead use an argument with reproducing kernels similar to that in \cite{Karaev}.

Since $S_1$ and $S_2$ are partial isometries on any closed invariant subspace of $H^2(\mathbb{D}^2)$, applying the second formula of \eqref{kernels} shows that the reproducing kernel of $O_2(\mathcal{M})$ is given by
$$
(1-\overline{\lambda_2}z_2)k_\lambda^{\mathcal{M}}(z),
$$
and since $O_2(\mathcal{M})$ is invariant under $S_1$ by \eqref{commutativeContainment}, applying the second formula of \eqref{kernels} again shows that the reproducing kernel for $O_1(\mathcal{M}) \cap O_2(\mathcal{M})$ is given by
\begin{equation} \label{RK_O_j}
(1-\overline{\lambda_1}z_1)(1-\overline{\lambda_2}z_2) k_\lambda^\mathcal{M}(z).
\end{equation}
If $O_1(\mathcal{M}) \cap O_2(\mathcal{M}) = \{0\}$, then 
$$
(1-\overline{\lambda_1}z_1)(1-\overline{\lambda_2}z_2) k_\lambda^\mathcal{M}(z) = 0 \Ra k_\lambda^\mathcal{M}(z) = 0,
$$
and so $\mathcal{M}$ is trivial.

If $O_1(\mathcal{M}) \cap O_2(\mathcal{M}) \neq \{0\}$, then we know that $O_1(\mathcal{M}) \cap O_2(\mathcal{M})$ is one-dimensional, and so its reproducing kernel will be given by $\overline{\varphi(\lambda)} \varphi(z)$ for some $\varphi$ with $H^2$-norm equal to $1$. From the previous arguments, we know that this $\varphi$ will in fact be inner. Put together, this means that

\begin{multline*}
(1-\overline{\lambda_1}z_1)(1-\overline{\lambda_2}z_2) k_\lambda^\mathcal{M} (z) = \overline{\varphi(\lambda)} \varphi(z) \\
\iff k_\lambda^\mathcal{M}(z) = \frac{\overline{\varphi(\lambda)} \varphi(z)}{(1-\overline{\lambda_1}z_1)(1-\overline{\lambda_2}z_2)}.
\end{multline*}

By \eqref{kernels} we recognize the right hand side of the last equation as the reproducing kernel of $\varphi H^2(\mathbb{D}^2)$. Since a Hilbert space is uniquely determined by its reproducing kernel, this finishes the proof.

\end{proof}

A somewhat informal, but perhaps helpful way to think about the derivation of formula \eqref{RK_O_j} for the reproducing kernel of $O_1(\mathcal{M}) \cap O_2(\mathcal{M})$ is to note that the orthogonal projection onto $O_j(\mathcal{M}) \subset \mathcal{M}$ is given by $P_j := (I - S_j S_j^*)$, and using that since $\{S_j \}_{j=1,2}$ are doubly commuting on $\mathcal{M}$, the projections $P_1$ and $P_2$ commute. As a consequence, the orthogonal projection onto $O_1(\mathcal{M}) \cap O_2(\mathcal{M})$ is given by $P_1 P_2$, and the reproducing kernel of $O_1(\mathcal{M}) \cap O_2(\mathcal{M})$ is
$$
P_1 P_2 k_\lambda^\mathcal{M}(z).
$$

Furthermore, it is worth noting that our argument for why
$$
O_1(\mathcal{M}) \cap O_2(\mathcal{M}) = \{0\} \Ra \mathcal{M} = \{ 0 \},
$$
is actually not very different from the argument given in \cite{Mandrekar}. For this Mandrekar essentially refers to the wandering subspace theorem from \cite{Wold}, and concludes that if $O_1(\mathcal{M}) \cap O_2(\mathcal{M}) = \{0\}$ then $\mathcal{M} = \{0 \}$, since clearly the wandering subspace $O_1(\mathcal{M}) \cap O_2(\mathcal{M})$ can't generate anything else. 

Note that the wandering subspace argument using reproducing kernels given above works exactly the same way in $n$ variables, and Mandrekar's argument for showing that the wandering subspace is one-dimensional also works in $n$ variables, although in that case the argument will become slightly more technical. But since \cite{Wold} only deals with a pair of doubly commuting isometries, Mandrekar's original theorem only concerns functions of two complex variables. However, more recently the main results in \cite{Wold} have been generalized by Sarkar in \cite{Wold_Sarkar} to deal with an $n$-tuple of pairwise doubly commuting isometries, and thus Mandrekar's original argument can be applied directly to show the corresponding theorem for $H^2(\mathbb{D}^n)$. Though, it is worth pointing out that Mandrekar's theorem on $H^2(\mathbb{D}^n)$ has already been proved by Seto in \cite{Seto} by using the Wold decompositions from \cite{Wold} in a slightly different way.

\section{Second proof}

In this section we give another proof of the second direction of Mandrekar's theorem, which is instead based on the proof idea of Beurling's theorem provided in \cite{Hoffman}.

We begin by proving a weaker version of the second direction of Mandrekar's theorem, where we assume that the origin is not a common zero for all functions in $\mathcal{M}$. The exact statement we prove is as follows.

\begin{thm} \label{weak_mandrekar}

If the shift operators $S_1$ and $S_2$ are \emph{doubly commuting} on an invariant subspace $\mathcal{M} \neq \{0\}$ of $H^2(\mathbb{T}^2)$, which has the additional property that it contains an element which does not vanish at the origin, then $\mathcal{M}$ is of the form $\varphi H^2$ for an inner function $\varphi$.
\end{thm}

\begin{proof}

As always we denote by $k_\lambda^{\mathcal{M}}(z)$ the reproducing kernel of $\mathcal{M}$. Now, denote by $\varphi(z)$ the reproducing kernel at the origin, $k_0^{\mathcal{M}}(z)$. Note that $k_0(z) \neq 0$ since we assume that some function in $\mathcal{M}$ does not vanish at the origin. We will show that $\varphi(z)$ is inner by showing that all Fourier coefficients of $|\varphi(z)|^2$ apart from the constant term are zero.

By the reproducing property of $\varphi$
\begin{equation} \label{proof2:orthogonal_to_monomials}
\int_{\mathbb{T}^2} |\varphi(z)|^2 z_1^{k} z_2^{n} |dz| = \langle z_1^k z_2^n \varphi(z),  \varphi(z) \rangle  = 0
\end{equation}
for all $(k,n) \neq (0,0)$ with $k,n \geq 0$, and by taking complex conjugates, we also see that
$$
\int_{\mathbb{T}^2} |\varphi(z)|^2 z_1^{k} z_2^{n} |dz|  = 0
$$
when $(k,n) \neq (0,0)$, and $k, n \leq 0$.

If $(k,n) \neq (0,0)$ with $k, n \geq 0$ then
$$
\int_{\mathbb{T}^2} |\varphi(z)|^2 z_1^{k} z_2^{-n} |dz| = \int_{\mathbb{T}^2} (\varphi(z) z_1^{k}) \overline{(\varphi(z) z_2^{n})} |dz| = \langle S_1^k \varphi(z), S_2^{n} \varphi(z) \rangle,
$$ 
and since $S_1$ and $S_2$ are assumed to be doubly commuting
$$
\langle S_1^k \varphi(z), S_2^{n} \varphi(z) \rangle = \langle \varphi(z), (S_1^*)^k S_2^{n} \varphi(z) \rangle = \langle \varphi(z), S_2^{n} (S_1^*)^k \varphi(z) \rangle = 0,
$$
where the last equality is again a consequence of the reproducing property of $\varphi$.

By using the same argument with $(-k,n)$ instead of $(k,-n)$, we finally see that 
$$
\int_{\mathbb{T}^2} |\varphi(z)|^2 z_1^{k} z_2^{n} |dz| = 0
$$
for all $(k,n) \neq (0,0)$, which shows that $|\varphi(z)|$ is constant almost everywhere on $\mathbb{T}^2$, and since $\varphi$ is normalized, this means that $\varphi$ is inner.

It remains to show that $\mathcal{M} = \varphi H^2(\mathbb{D}^2)$.

Since $\varphi$ has constant modulus on the boundary and since the polynomials are dense in $H^2(\mathbb{D}^2)$, we have that
$$
\varphi H^2(\mathbb{D}^2) = \overline{\{\varphi(z) p(z): p(z) \in \mathbb{C}[z_1,z_2] \}} \subset \mathcal{M}.
$$
Now let $f \in \mathcal{M}$ be orthogonal to all elements in $\{\varphi(z) p(z): p(z) \in \mathbb{C}[z_1,z_2] \}$. We will show that $f=0$ by showing that all the Fourier coefficients of $\varphi \overline{f}$ are zero.

By the orthogonality assumption on $f$
$$
\int_{\mathbb{T}^2} \varphi(z) \overline{f(z)} z_1^k z_2^n |dz| = 0,
$$ 
for all $n,k \geq 0$. Furthermore, by again using the reproducing property of $\varphi$, and since $S_1$ and $S_2$ are doubly commuting on $\mathcal{M}$, we have that
\begin{multline*}
\int_{\mathbb{T}^2} \varphi(z) \overline{f(z)} z_1^k z_2^{-n} |dz| = \int_{\mathbb{T}^2} \varphi(z)z_1^k \overline{f(z) z_2^n} |dz| \\
= \langle z_1^k \varphi(z), z_2^n f(z) \rangle = \langle \varphi(z), S_2^n (S_1^*)^k f(z) \rangle = 0,
\end{multline*}
for $n \geq 1$ and $k \geq 0$. The same argument for $(-k,n)$ instead of $(k,-n)$ shows that
$$
\int_{\mathbb{T}^2} \varphi(z) \overline{f(z)} z_1^{-k} z_2^{n} |dz| = 0
$$
when $n \geq 0$ and $k \geq 1$.

It remains to show that 
\begin{equation} \label{proof2:last_orthogonality}
\langle \varphi, f(z) z_1^k z_2^n \rangle = \int_{\mathbb{T}^2} \varphi(z) \overline{f(z)} z_1^{-k} z_2^{-n} |dz| = 0
\end{equation}
for $k,n \geq 1$.

But this is an immediate consequence of the reproducing property of $\varphi$ since $f(z) z_1^k z_2^n$ belong to $\mathcal{M}$ and vanish at the origin for all $k,n \geq 1$.

It follows that all Fourier coefficients of $\varphi \overline{f}$ vanish, and so $\varphi \overline{f} = 0$. Since $|\varphi| = 1$ a.e. on $\mathbb{T}^2$, this means that $f = 0$ a.e. on $\mathbb{T}^2$, and thus $f$ is identically equal to zero. 
\end{proof}

For functions of one variable, the general case --- in which all function in $\mathcal{M}$ might have a common zero at the origin --- is easily reduced to the case above by, if necessary, simply factoring out $z^k$ for some $k \geq 1$. For functions of several variables this is no longer possible since there is no canonical factor corresponding to zeros at the origin. Instead, we modify the argument given above as follows.

\begin{proof}[Second proof of the second implication of Theorem~\ref{mandrekar}]

We will show that if $S_1$ and $S_2$ are doubly commuting on $\mathcal{M}$, then $\mathcal{M} = \varphi H^2(\mathbb{D}^2)$ for some inner function $\varphi$. 

If there is some $f \in \mathcal{M}$ which does not vanish at the origin, we can just apply Theorem~\ref{weak_mandrekar}, and then there is nothing more to be done. 

Now suppose $d \geq 1$ is the smallest integer such that all partial derivatives of total degree less than $d$ of all functions $f \in \mathcal{M}$ vanish at the origin. That is, for all $j_1, j_2 \in \mathbb{N}$ with $j_1+j_2 < d$
\begin{equation} \label{deriv_low_order_vanish}
\left( \frac{\partial^{j_1 + j_2}f}{\partial z_1^{j_1} \partial z_2^{j_2}}  \right)(0,0) = 0
\end{equation}
for all $f \in \mathcal{M}$. We may assume that $d < \infty$ since otherwise $\mathcal{M} = \{0 \}$.  

Let $(d_1, d_2) \in \mathbb{N}^2$ with $d_1+d_2 = d$ be any pair of integers such that 
$$
\left( \frac{\partial^{d_1 + d_2}f}{\partial z_1^{d_1} \partial z_2^{d_2}}  \right)(0,0) \neq 0
$$
for some $f \in \mathcal{M}$.
Consider the bounded linear functional $E^{(d_1,d_2)}_0$ on $H^2(\mathbb{D}^2)$ defined by 
$$
E^{(d_1,d_2)}_0: f \mapsto \left( \frac{\partial^{d_1 + d_2}f}{\partial z_1^{d_1} \partial z_2^{d_2}}  \right)(0,0).
$$
That this functional is indeed bounded on $H^2(\mathbb{D}^2)$ is clear since it maps a function $f$ to a fixed constant multiple of its Fourier coefficient with index $(d_1, d_2)$.  

By the Riesz representation theorem there exists a unique function, which we will denote by $k^{(d_1,d_2)}_0$, such that
$$
E^{(d_1,d_2)}_0(f) = \left\langle f, k^{(d_1,d_2)}_0 \right\rangle
$$
for all $f \in H^2(\mathbb{D}^2)$. This is kind of a reproducing kernel at the origin, only it gives the value for the $(d_1,d_2)$:th partial derivative instead of for the function.

Just as for the ordinary reproducing kernel, we have that $P^\mathcal{M} k^{(d_1,d_2)}_0 \in \mathcal{M}$ is the unique function in $\mathcal{M}$ such that
$$
E^{(d_1,d_2)}_0(f) = \left\langle f, P^{\mathcal{M}} k^{(d_1,d_2)}_0 \right\rangle
$$
for all $f \in \mathcal{M}$. From now on we will denote the normalized function $P^{\mathcal{M}} k^{(d_1,d_2)}_0 / \| P^{\mathcal{M}} k^{(d_1,d_2)}_0 \|$ by $\varphi$. 

Since $\varphi$ reproduces a rescaling of the $(d_1,d_2)$:th partial derivative at the origin and since equation \eqref{deriv_low_order_vanish} holds for all $f \in \mathcal{M}$, we have that
$$
\int_{\mathbb{T}^2} |\varphi(z)|^2 z_1^{k} z_2^{n} |dz| = \langle z_1^k z_2^n \varphi(z),  \varphi(z) \rangle  = 0,
$$
for all $(n,k) \neq (0,0), n,k \in \mathbb{N}$. That is $\varphi$ satisfies equation \eqref{proof2:orthogonal_to_monomials} from the proof of Theorem~\ref{weak_mandrekar}.

In fact, as a consequence of the reproducing property of $\varphi$ and the fact that \eqref{deriv_low_order_vanish} holds for all $f \in \mathcal{M}$, we have that 
\begin{equation} \label{deriv_kernel_rep_prop}
\langle \varphi, f(z) z_1^n z_2^k \rangle = 0
\end{equation}
for all $f \in \mathcal{M}$ and all $n,k \in \mathbb{N}$ with $(n,k) \neq (0,0)$. To see that \eqref{deriv_kernel_rep_prop} holds, note that when we evaluate the terms of the partial derivatives of $f(z) z_1^n z_2^k$ 
 at the origin, either they will vanish because of a monomial factor $z_1^l z_2^m$ still being left after differentiation, or they will vanish because of \eqref{deriv_low_order_vanish}.

Now in order to finish the proof we can just use the proof of Theorem~\ref{weak_mandrekar} from equation \eqref{proof2:orthogonal_to_monomials} verbatim, if we just replace any reference to "the reproducing property of $\varphi$" with a reference to equation \eqref{deriv_kernel_rep_prop}. 

\end{proof}

The function $\varphi$ used above can be obtained as the unique minimizer of a suitable extremal problem. When obtained in this way, one instead shows that equation \eqref{proof2:orthogonal_to_monomials} holds through a variational argument. This is of course more complicated than the argument given above, but in the one variable setting this approach has been successfully applied to extend results whose usual proofs rely heavily on Hilbert space techniques to Banach spaces like $H^p(\mathbb{D})$ and $A^p(\mathbb{D})$. In this specific context though it might be worth to point out that Mandrekar's theorem for $H^p(\mathbb{D}^2)$ for $p \geq 1$ has already been proved by Redett in \cite{Redett}. The argument used in that article is a modification of the idea of considering the intersection of the invariant subspace with $H^2(\mathbb{D}^2)$, applying Mandrekar's theorem, and using a density argument.

\end{document}